\documentclass[12pt]{amsart}
\usepackage{Baskervaldx, amssymb, graphicx, caption, subcaption, amsmath, array, pdfsync, tikz, tikz-cd, color, url, float, enumerate, geometry, mathrsfs, verbatim, colonequals}
\usepackage[baskervaldx,cmintegrals,bigdelims,vvarbb]{newtxmath}
\usepackage[cal=cm]{mathalfa}
\usepackage[all]{xy}

\usetikzlibrary{calc}

\title{Fano hypersurfaces with arbitrarily large degrees of irrationality\vspace{-1ex}}
\author{Nathan Chen and David Stapleton}

\usepackage[hyperfootnotes=false]{hyperref}
\hypersetup{
  colorlinks= true,
  linkcolor=[RGB]{0,89,42},
  urlcolor=[RGB]{0,89,42},
  citecolor=[RGB]{0,89,42}
}

\newcommand\blfootnote[1]{%
  \begingroup
  \renewcommand\thefootnote{}\footnote{#1}%
  \addtocounter{footnote}{-1}%
  \endgroup
}

\newtheorem{thm}{Theorem}
\numberwithin{thm}{section}

\newtheorem{lem}[thm]{Lemma}

\newcommand{\hr}[2]{\hyperref[#1]{#2}}

\newtheorem{Lthm}{Theorem}

\newtheorem{Lcor}[Lthm]{Corollary}
\newtheorem{Lprop}[Lthm]{Proposition}

\theoremstyle{definition}
\newtheorem{remark}[thm]{Remark}

\newtheorem{example}[thm]{Example}
\newtheorem{definition}[thm]{Definition}
\newtheorem{construction}[thm]{Construction}

\def\P{{\mathbb P}}
\def\Frac{{\mathrm{Frac}}}
\def\C{{\mathbb C}}

\def\O{{\mathcal O}}
\def\M{{\mathcal M}}

\def\pf{{\mathfrak{p}}}

\def\Oc{{\mathcal O}}

\def\X{{\mathscr X}}

\def\Spec{{\mathrm{Spec}}}
\def\kapbar{{\overline{\kappa}}}
\def\etabar{{\overline{\eta}}}

\def\min{{\mathrm{min}}}
\def\max{{\mathrm{max}}}

\def\dra{{\dashrightarrow}}
\def\ra{{\rightarrow}}

\def\cl{{\colon}}

\def\deg{{\mathrm{deg}}}

\def\irr{{\mathrm{irr}}}

\def\min{{\mathrm{min}}}

\def\bir{{\simeq_{\mathrm{bir}}}}
\def\Tr{{\mathrm{Tr}}}
\def\dim{{\mathrm{dim}}}

\def\Graph{{\mathrm{Graph}}}

\def\Oc{{\mathcal O}}

\begin{document}
\maketitle
\vspace{-20pt}
\section*{Introduction}

\thispagestyle{empty}

There has been a great deal of interest in studying questions of rationality of various flavors. Recall that an $n$-dimensional variety $X$ is \textit{rational} if it is birational to $\P^n$ and \textit{ruled} if it is birational to $\P^1\times Z$ for some variety $Z$. Let $X\subset \P^{n+1}$ be a degree $d$ hypersurface. In \cite{Kollar95}, Koll\'ar proved that when $X$ is very general and $d \geq (2/3)(n+3)$ then $X$ is not ruled (and thus not rational). Recently these results were generalized and improved by Totaro \cite{Totaro} and subsequently by Schreieder \cite{Schreieder18}. Schreieder showed that when $X$ is very general and $d \geq \log_{2}(n)+2$ then $X$ is not even stably rational. In a positive direction, Beheshti and Riedl \cite{BR19} proved that when $X$ is smooth and $n\ge 2^{d!}$ then $X$ is at least \textit{unirational}, i.e., dominated by a rational variety. \blfootnote{\hspace{-12pt}The first author's research is partially supported by the National Science Foundation under the Stony Brook/SCGP RTG grant DMS-1547145.}

Given a variety $X$ whose non-rationality is known, one can ask if there is a way to measure ``how irrational" $X$ is. One natural invariant in this direction is the \textit{degree of irrationality}, defined as
\[ \irr(X) \coloneqq \min \{ \delta > 0 \mid \exists \text{ degree $\delta$ rational dominant map } X \dashrightarrow \P^{n} \}. \]
For instance, Bastianelli, De Poi, Ein, Lazarsfeld, and Ullery \cite{BDELU17} computed the degree of irrationality for very general hypersurfaces of degree $d \geq 2n+1$ by using the positivity of the canonical bundle. Naturally, one is tempted to ask what can be proved about hypersurfaces with negative canonical bundle.

Let $X\subset \P^{n+1}$ be a smooth hypersurface of degree $d$. We define the \textit{codegree} of $X$ to be the quantity
\[
e \colonequals n+2-d,
\]
so that $K_X = -eH$. When $X$ is Fano (in other words $e>0$), this coincides with the \textit{Fano index}. Our main result gives the first examples of Fano varieties with arbitrarily large degree of irrationality.

\begin{Lthm}\label{IrrHyp}
Let $X_{n,e}$ be a very general hypersurface over $\C$ of dimension $n$ and codegree $e$. For any fixed $e$, there exists $N>0$ such that for all $n>N$
\[
\irr(X_{n,e})\ge \frac{1}{4}\sqrt{n}.
\]
\end{Lthm}

\noindent In fact, we prove the stronger statement that the minimal degree map from $X_{n,e}$ to a ruled variety is bounded from below by $\sqrt{n}/4$. To the best of our knowledge, these give the first examples of rationally connected varieties $X$ with $\irr(X)\ge 4$. (Irrational Fano varieties $X$ have $\irr(X)\ge 2$, and rational covers of degree 2 always admit birational involutions, so by the Noether--Fano method a general quartic threefold $X$ satisfies $\irr(X)=3$.)

The proof of Theorem A proceeds by extending ideas from Koll\'ar's paper \cite{Kollar95}, using a specialization to characteristic $p>0$. This involves two main additions to the arguments in \cite{Kollar95}. First, we use positivity considerations involving separation of points to show that the hypersurfaces constructed by Koll\'ar do not admit low degree maps to ruled varieties. Our main technical result then asserts that such mappings, if they exist, behave well in families. Specifically, given a family of projective varieties over the spectrum of a DVR, we prove that the minimal degree of a rational map to a ruled variety can only drop upon specialization. To be more precise, let $T$ be the spectrum of a DVR with generic point $\eta$ and residue field $\kappa$, and let $\X_T$ be an integral normal scheme which is projective over $T$ such that $\X_\eta$ is geometrically integral. With this set-up in mind, we have:

\begin{Lthm}\label{RuledSpec}
If $\X_\eta$ admits a dominant and generically finite rational map to a ruled variety with degree $\le d$, then so does every component $\X'_\kappa \subset \X_\kappa$.
\end{Lthm}

\noindent A similar result holds for the geometric generic fiber (see Theorem~\ref{DegMaps}).

We give other applications of Theorem~\ref{RuledSpec} concerning the behavior of the degree of irrationality in certain families. The last few years have seen major progress in understanding the behavior of rationality and stable rationality in families. Hassett, Pirutka, and Tschinkel \cite{HPT18} showed that there are families of varieties where there is a dense set of rational fibers, but the very general member is irrational. Nicaise and Shinder \cite{NicShi19} (resp. Kontsevich and Tschinkel \cite{KonTsch19}) established that stable rationality (resp. rationality) specializes in smooth projective families. The behavior of unirationality and the degree of irrationality in families is understood to a lesser extent. Applying Theorem~\ref{RuledSpec}, we show that in certain families the degree of irrationality can only drop upon specialization.

\begin{Lprop}\label{ExSpec}
Let $\pi\cl \X\ra T$
be a smooth projective family of complex varieties over a smooth irreducible curve $T$ with marked point $0\in T$. Assume that a very general fiber $\X_t$ is either \leavevmode
\begin{enumerate}
    \item a surface with $H^{1,0}(\X_t)=0$, or
    \item a simply-connected threefold with $\omega_{\X_t} \cong \O_{\X_t}$ (i.e., a strict Calabi-Yau threefold).
\end{enumerate}
If $\irr(\X_t)\le d$, then $\irr(\X_0)\le d$.
\end{Lprop}

\noindent By work of the first author \cite{Chen19}, it is known that a very general abelian surface $A$ has $\irr(A)\le 4$. From Proposition~\ref{ExSpec}, we are able to deduce:

\begin{Lcor}\label{IrrAb}
Every complex abelian surface $A$ has $\irr(A)\le 4.$
\end{Lcor}

In \S1, we prove a generalized version of Theorem~\ref{RuledSpec}, as well as Proposition~\ref{ExSpec} and Corollary~\ref{IrrAb}. In \S2, we prove Theorem~\ref{IrrHyp}. Throughout the paper, by \textit{variety} we mean an integral scheme of finite type over a field (not necessarily algebraically closed). By very general, we mean away from the union of countably many proper closed subsets. If $\X$ is a scheme over $T$ and $k$ is a field with a map $\Spec(k)\ra T$, by abuse of notation we write
\[
\X_k \colonequals \X\times_T\Spec(k).
\]

\noindent \textbf{Acknowledgments.} We are grateful for valuable conversations and correspondences with Iacopo Brivio, Lawrence Ein, Fran\c cois Greer, Kiran Kedlaya, J\'anos Koll\'ar, Robert Lazarsfeld, James M\textsuperscript{c}Kernan, John Ottem, John Sheridan, Jason Starr, Burt Totaro, and Ruijie Yang.

\section{Maps to ruled varieties specialize}

The goal of this section is to prove a slightly more general version of Theorem~\ref{RuledSpec}. Let $T=\Spec(A)$ be the spectrum of a DVR with uniformizer $t$, fraction field $\eta = \Frac(A)$, and residue field $\kappa=A/t$. Let $\etabar$ and $\kapbar$ be their respective algebraic closures. Assume that $\X_{T} \rightarrow T$ is a flat projective morphism, and $\X_{T}$ is a normal integral scheme such that $\X_\eta$ is geometrically integral.

\begin{thm}\label{DegMaps} With the set-up above, \leavevmode
\begin{enumerate}
    \item If $\X_{\eta}$ admits a dominant, generically finite rational map to a ruled variety
    \[
    \varphi \cl \X_{\eta} \dashrightarrow Z_\eta \times \P^{1}
    \]
    with $\deg(\varphi)\le d$, then so does every component $\X_{\kappa}' \subset \X_{\kappa}$.
    \item Suppose that $\X'_\kappa\subset \X_\kappa$ is a component that is geometrically reduced and geometrically irreducible. If $\X_{\overline{\eta}}$ admits a dominant rational map to a ruled variety
    \[
    \varphi \cl \X_{\etabar} \dashrightarrow Z_\etabar \times \P^{1}
    \]
    with $\deg(\varphi)\le d$, then so does $\X'_\kapbar$.
\end{enumerate}
\end{thm}

\begin{proof}
First we prove (1). We may assume $Z$ is the generic fiber of a reduced and irreducible projective scheme $Z_T$. As $\X_T$ is normal and the schemes are projective over $T$, the rational map $\varphi$ extends to a map $\varphi_T$ which is defined on all codimension 1 points. By an argument of Abhyankar and Zariski \cite[Lem. 2.22]{Kollar13} for any codimension 1 point $\delta\in X$, there is a birational morphism $\mu\cl R_T\ra Z_T\times \P^1$ so that the induced rational map
\[
\varphi'_T\cl \X_T\dra R_T
\]
satisfies $\varphi'_T(\delta)$ is a codimension 1 point and $R_T$ is regular at the generic point of $\varphi'_T(\delta)$. When $\delta$ is the generic point of $X'_\kappa$, then the closure of the image $R'_\kappa:=\overline{\varphi'_T(\delta)}\in R_\kappa$ satisfies the hypotheses of Matsusaka's Theorem (see \cite[Thm. IV.1.6]{Kollar96}). Therefore $R'_\kappa$ is ruled, and we have produced a dominant finite rational map $\varphi'_\kappa\cl \X_\kappa'\dra R'_\kappa$. It is not hard to see that $\deg(\varphi_\kappa')\le d$.


Now we show (1) implies (2). There is a finite field extension $\eta \subset L\subset \etabar$ such that $\varphi$ is defined over $L$, i.e. there is a map $\X_L\dra Z_L\times \P^1$ whose base change to $\etabar$ is $\varphi$. Let $B\subset L$ be the integral closure of $A$. Let $\pf\subset B$ denote the ideal of a closed point over $(t=0)$ with residue field $\kappa_B=B/\pf$. Thus we have a map of DVRs $A\ra B_\pf$. Let $S=\Spec(B_\pf)$ and let $\X_S \colonequals \X_T\times_T S$ be the base change of $\X_T$ to $S$. (1) implies that every component of the central fiber of the normalization $(\X_S^{\mathrm{norm}})_{\kappa_B}$ admits a map to a ruled variety with degree bounded by $d$.

It remains to show that for any component $\X'_\kappa\subset \X_\kappa$ which is geometrically reduced and geometrically irreducible (as in the statement of the theorem), the normal locus of $\X_S$ contains the generic point of the divisor
\[
\X'_{\kappa_B} \colonequals \X'_\kappa \times_\kappa \kappa_B \subset \X_S.
\]
The assumption that $\X'_\kappa$ is geometrically reduced implies that $\X_S$ is regular at the generic point of $\X'_{\kappa_B}$. So in particular, $\X_{S}$ is normal at the generic point of $\X_{\kappa_B}'$.
\end{proof}

Now we prove two results about the behavior of the minimal degree map to a ruled variety.

\begin{lem}\label{BaseChangeMaps}
Let $k\subset L$ be a field extension with both fields algebraically closed. If $X$ is a variety of dimension $n$ over $k$, then $X$ admits a rational map of degree $d$ to a ruled variety over $k$ $\iff$ $X_L \coloneqq X\times_k L$ admits a rational map of degree $d$ to a ruled variety over $L$.
\end{lem}

\begin{proof}
By standard arguments, there is a countable union $\Graph_k^d(X_k)$ of quasiprojective schemes which parametrizes closures of graphs
\[
\Gamma\subset X\times_k (Z\times_k \P^1)
\]
of rational maps $X \dashrightarrow Z\times \P^1$ to a ruled projective $k$-variety of degree $d$. As a consequence, $X$ admits a degree $d$ map to a ruled variety $\iff \Graph_k^d(X)\ne \emptyset$ (as $k$ is algebraically closed). The point is
\[
\Graph_L^d(X_L) = \Graph_k^d(X_k)\times_k L.
\]
Thus, if $X_L$ has a degree $d$ map to a ruled variety, then $X_k$ has a degree $d$ map to a ruled variety (which is not to say that the original map is defined over $k$).
\end{proof}

We thank Fran\c cois Greer and Burt Totaro for suggesting the following lemma.

\begin{lem}\label{VGen}
Working over $\C$, let $\X \rightarrow T$ be a family of varieties over a smooth curve $T$ with geometric generic point $\etabar$. If a very general fiber $\X_{t}$ admits a rational map of degree $d$ to a ruled variety then so does $\X_\etabar$.
\end{lem}

\begin{proof}
By \cite[Lem. 2.1]{Vial13}, there is a field isomorphism of $\etabar$ with $\C$ such that $\X_\etabar$ becomes isomorphic to $\X_{t}$ as abstract schemes. The lemma follows by composing this isomorphism with the rational map of degree $d$ from $\X_{t}$ to a ruled variety.
\end{proof}

Using Theorem~\ref{DegMaps}, we can now prove Proposition~\ref{ExSpec} and Corollary~\ref{IrrAb}.

\begin{proof}[Proof of Proposition~\ref{ExSpec}]
Let
\[
\pi \cl \X \rightarrow T
\]
be a family of smooth projective varieties as in Proposition~\ref{ExSpec}. By Theorem~\ref{DegMaps} and Lemma~\ref{VGen}, it follows that the special fiber $\X_{0}$ admits a dominant and generically finite rational map
\[
\varphi \cl \X_{0} \dashrightarrow Y \bir Z \times \P^{1}
\]
such that $\deg(\varphi) \leq d$, and $Z$ is smooth. Consider the MRC fibration of $Z$, given by
\[
\psi\cl Z\dra B
\]
where $B$ is smooth. Then $B$ has dimension 0, 1, or 2. We treat each dimension separately. In the case $B$ has dimension 0, then $Z$ is in fact rational, and we have $\irr(\X_0)\le d$.

Now we rule out the cases where $\dim(B)=1$ or 2. If $\dim(B)=1$, then $B$ must have positive geometric genus. But there are no dominant rational maps from $\X_0$ to such a curve (as $h^{1,0}(\X_0)=0$). If $\dim(B)=2$, then we are in the case (2) of Proposition~\ref{ExSpec}, so $\X_0$ is a strict Calabi-Yau threefold. It follows that $B$ is a Kodaira dimension 0 surface (see \cite{Birkar09}). By the classification of surfaces there is a proper \'etale cover $B'\ra B$ such that $\omega_{B'}\cong\Oc_{B'}$. As $\X_0$ is simply connected, the map from $\X_0$ to $B$ factors through $B'$. But this contradicts the fact that strict Calabi-Yau threefolds have $h^{p,0}=0$ for $p=1,2$.
\end{proof}

\begin{proof}[Proof of Corollary~\ref{IrrAb}]
The quotient map from $A$ to the Kummer surface $A/\pm 1$ has degree 2, so it suffices to prove $\irr(A/\pm 1)= 2$. The main result of \cite{Chen19} can be rephrased by saying that a very general Kummer surface has degree of irrationality 2. It follows that one can put $A/\pm 1$ in a family over a curve such that the very general member has degree of irrationality equal to 2. Taking a simultaneous resolution gives a family of K3 surfaces, so by Proposition~\ref{ExSpec} every member of the family has degree of irrationality equal to 2.
\end{proof}

\section{Irrationality of Fano hypersurfaces}

The goal of this section is to prove Theorem~\ref{IrrHyp}. For a hypersurface $X \subset \P^{n+1}$ of degree $d$, recall from the introduction that the codegree of $X$ is $e \colonequals n+2-d$. We follow an idea of Koll\'ar's \cite{Kollar95}, which was used to prove non-ruledness of certain Fano hypersurfaces. Koll\'ar reduces to positive characteristic and observes that if a smooth projective variety $X$ is ruled, then no sheaf of $i$-forms $\wedge^i \Omega_{X}$ can contain a big line bundle. The main observation in this section (Lemma~\ref{InjLB}) is that if $X$ admits a separable rational map of degree $d$ to a ruled variety, then no sheaf of $i$--forms $\wedge^i \Omega_{X}$ can contain a line bundle which separates $2d$ points on an open set. This allows us to directly apply Koll\'ar's degeneration argument (albeit in a different degree range) to deduce Theorem~\ref{IrrHyp}.

\begin{definition}
Let $X$ be a variety over an algebraically closed field $k$ and let $\M$ be a line bundle on $X$. We say that $H^0(X,\M)$ \textit{separates $\ell$ points on an open set} if there is an open set $U\subset X$ such that for any distinct $\ell$ points $x_1, \ldots, x_{\ell}\in U$, there is a section $s\in H^0(X,\M)$ which vanishes on $x_1, \ldots,x_{\ell-1}$ but not $x_{\ell}$.
\end{definition}

\begin{example}\label{ExSep}
Let $X$ and $Y$ be projective varieties of dimension $n$ over an algebraically closed field $k$ and suppose there is a map
\[
\psi\cl X\ra Y
\]
which is a dominant, generically finite, purely inseparable morphism. Then over the open set $U\subset Y$ where $\psi$ is finite, there is a bijection on $k$ points
\[
\psi^{-1}(U)(k)\cong U(k).
\]
As a consequence, if there is a line bundle $\M$ on $Y$ such that $H^0(Y,\M)$ separates $\ell$ points on an open set, then $H^0(X,\psi^*(\M))$ separates $\ell$ points on an open set in $X$.
\end{example}

\begin{lem}\label{InjLB}
Let $X$ be a projective variety over an algebraically closed field and $\M$ a line bundle on $X$ such that $H^0(X,\M)$ separates $2\delta$ points on an open set. Suppose that there is an injection
\[
\M\hookrightarrow \wedge^i\Omega_{X}
\]
for some $i>0$. If there is a dominant, separable, generically finite rational map
\[
\varphi\cl X\dra Z\times \P^1
\]
to a ruled variety, then $\deg(\varphi)\ge \delta+1$.
\end{lem}

\begin{proof}
Let $b \coloneqq \deg(\varphi)$. Without loss of generality, we may assume that $X$ and $Z$ are normal. Let $\Gamma$ be the normalization of the closure of the graph of $\varphi$. This gives two regular maps $\mu$, $\psi$ which make the diagram commute:
\[
\begin{tikzcd}
&\Gamma\arrow[dl,swap,"\mu"]\arrow[dr,"\psi"]&\\
X\arrow[rr,dashed,"\varphi"]&&Z\times \P^1.
\end{tikzcd}
\]
Note that $\mu$ is birational. By Example~\ref{ExSep} we have that $\mu^*(\M)$ is a line bundle which separates $2\delta$ points on an open set in $\Gamma$ and injects into $\wedge^i\Omega_\Gamma$.

Recall that there is a trace map
\[
\Tr^i_\psi\cl H^0(\Gamma,\wedge^i\Omega_\Gamma)\ra H^0(Z\times \P^1,(\wedge^i\Omega_{Z\times \P^1})^{\vee \vee})
\]
(see \cite[Prop. 3.3]{dJS04} and \cite{dJS08}) which extends the usual trace map for finite morphisms in characteristic 0. Over the dense open set in $Y$ where $\psi$ is \'etale, $\Tr_{\psi}^{i}$ corresponds to the sum over fibers. Let $\pi_1$ (resp. $\pi_2$) denote the projection of $Z\times \P^1$ onto $Z$ (resp. $\P^1$). By a computation,
\[
(\wedge^i\Omega_{Z\times \P^1})^{\vee \vee}\cong (\wedge^i(\pi_1^*\Omega_Z))^{\vee\vee}\oplus\left( (\wedge^{i-1}(\pi_1^*\Omega_Z))^{\vee\vee}\otimes \pi_2^*(\Omega_{\P^1})\right).
\]
So any section of $H^0(Z\times\P^1,(\wedge^i\Omega_{Z\times \P^1})^{\vee \vee})$ is necessarily constant along $\{ z \} \times \P^1$ for every $z \in Z$.

However, if we take a general point $z\in Z$ and two general points $z_1,z_2\in z\times \P^1$, then the preimage of these points will consist of $2b$ distinct points
\[
\psi^{-1}(\{z_1,z_2\})=\{\gamma_1,\dots,\gamma_{2b}\}.
\]
If $b\le \delta$, then there will be a section
\[
\alpha\in H^0(\Gamma,\mu^*(\M))
\]
which vanishes on $\gamma_1,\dots,\gamma_{2b-1}$ but does not vanish on $\gamma_{2b}$. Tracing $\alpha$ as an $i$-form gives a section
\[
\Tr^i_\psi(\alpha) \in H^0(Z\times \P^1,(\wedge^i\Omega_{Z\times \P^1})^{\vee \vee})
\]
which vanishes on $z_1$ but not on $z_2$, yielding a contradiction. Therefore $b\ge \delta+1$.
\end{proof}

\begin{construction}
Now we recall Koll\'ar's degeneration argument \cite[\S 5]{Kollar95}, which is attributed to Mori. Let $R$ be a DVR with an algebraically closed residue field $\kappa$ of characteristic $p$ and fraction field $\eta$ which is countable of characteristic 0. Koll\'ar shows that over $T=\Spec(R)$, one can construct an irreducible normal variety
\[
\pi\cl\X\ra T
\]
such that $\X_\eta\subset \P^{n+1}_\eta$ is a hypersurface of degree $pd$ and $\X_\kappa$ is a reduced degree $p$ inseparable cover of a smooth degree $d$ hypersurface $Y\subset \P^{n+1}_\kappa$ with ``simple" singularities. Koll\'ar gives an explicit resolution of singularities of $\X_\kappa$:
\[
\begin{tikzcd}
\X'_\kappa\arrow[r,"\mu"]\arrow[rr,swap,bend right=30,"\nu"] &\X_\kappa\arrow[r]&Y
\end{tikzcd}
\]
and shows that $\M \colonequals \nu^*(\Oc_{\P_\kappa^{n+1}}(pd+d-n-2))$ injects into $\wedge^{n-1}\Omega_{\X'_\kappa}$.
\end{construction} 

Having recalled Koll\'ar's construction, we are ready to prove Theorem~\ref{IrrHyp}.

\begin{proof}[Proof of Theorem~\ref{IrrHyp}]
To start we consider the case of hypersurfaces of degree $pd$ or codegree $e=n+2-pd$ degenerating to $p$-fold covers of degree $d$ hypersurfaces, as in Koll\'ar's construction above. We know that $\nu$ is purely inseparable, and the sections of $H^0(\P^{n+1},\Oc_{\P^{n+1}_\kappa}(d-e))$ separate every set of $d-e$ points. So by Example~\ref{ExSep} we see that $H^0(\X'_\kappa ,\M)$ separates $d-e$ general points. By Lemma~\ref{InjLB} it follows that $\X_\kappa$ admits no separable rational maps to a ruled variety with degree $\le \lfloor (d-e)/2\rfloor$.

Let us assume that $(d-e)\ge 2p-2$. Using the fact that $p$ divides the degree of any inseparable map, it follows that there are no rational maps at all from $\X_\kappa$ to a ruled variety with degree $\le p-1$. Thus Theorem~\ref{DegMaps} implies  $\X_\etabar$ admits no maps of degree $\le p-1$ to a ruled variety. By base changing $\etabar$ to $\C$ and applying Lemma~\ref{BaseChangeMaps} it follows that if $X\subset \P_\C^{n+1}$ is a very general complex hypersurface of degree $pd$ with $d-e\ge 2p-2$ then every rational map from $X$ to a ruled variety has degree $\ge p$.

Now we must deal with the case where $X$ is very general but its degree is not equal to $pd$ for some ``useful" prime $p$. Assume that
\[
n+2=(pd+f)+e.
\]
where $f$ is the remainder of $n+2-e$ modulo $p$. We claim that if $d-e-f\ge 2p-2$ then any map from $X$ to a ruled variety has degree $\ge p$. This can be proved by degenerating $X$ to a union of a very general degree $pd$ hypersurface and $f$ hyperplanes. The claim then follows from Theorem~\ref{DegMaps}(2) and the previous paragraph.

Therefore if $X$ has codegree $e$ in $\P^{n+1}$, we want to consider prime numbers $p$ such that there exist integers $d$ and $f$ with the property that
\[
n+2=(pd+f)+e\text{ and }d-e-f\ge 2p-2. \tag{$\ast$}
\]
Thus we have shown that
\[
\irr(X)\ge \max\{p \text{ prime} \mid p \text{ satisfies }(\ast)\}.
\]
It remains to estimate the right hand side. Suppose a prime $p$ satisfies
\[
4p^2\le n+2\text{ and }e\le p+1.
\]
One can verify that property $(\ast)$ holds. In other words,
\[
\irr(X)\ge \max\left(\{1\}\cup \{p\text{ prime} \mid 4p^2\le n+2 \text{ and }e\le p+1\}\right).
\]
By Bertrand's postulate, there is always a prime between any natural number $m$ and $2m$. Therefore, we can give the lower bound
\[
\irr(X)\ge \frac{\sqrt{n+2}}{4}
\]
as soon as the right hand side is at least $e-1$.
\end{proof}

\begin{remark}
In Theorem~\ref{IrrHyp} one can take $N=(4e-4)^2-2.$
\end{remark}

\begin{remark}
We note that the above bound is not optimal. One might hope that the bound on $\irr(X)$ can be improved to a linear bound. If one adapts the second paragraph of the proof of Theorem~\ref{IrrHyp}, the first new bounds occur using $p=5$, $n=35$, and $e=1$. In other words, a very general degree $35$ complex hypersurface $X\subset\P^{35}$ has $\irr(X)\ge 4$.
\end{remark}

%
%

\bibliographystyle{siam} 
\bibliography{Biblio}

\end{document}